\documentclass[11pt]{article}

\usepackage{a4wide}

\usepackage[english]{babel}
\usepackage{amsmath,amsthm}
\usepackage{amsfonts}
\usepackage{amssymb}
\usepackage{enumitem}
\usepackage{mathtools}
\usepackage{graphicx}
\usepackage{wrapfig}
\usepackage{tikz-cd}
\usepackage[colorlinks=true,linkcolor=blue]{hyperref}

\newcommand{\N}{\mathbb{N}}
\newcommand{\Z}{\mathbb{Z}}
\newcommand{\Q}{\mathbb{Q}}

\newcommand{\F}{\mathbb{F}}
\newcommand{\Fp}{\mathbb{F}_p}
\newcommand{\Zp}{\mathbb{Z}_p}
\newcommand{\Qp}{\mathbb{Q}_p}

\newcommand{\g}{\mathfrak{g}}
\newcommand{\kk}{\mathfrak{k}}
\newcommand{\q}{\mathfrak{q}}
\newcommand{\n}{\mathfrak{n}}
\newcommand{\U}{\mathcal{U}}

\DeclareMathOperator{\dd}{d}

\DeclareMathOperator{\id}{id}

\DeclareMathOperator{\Hom}{Hom}

\DeclareMathOperator{\HH}{H}

\DeclareMathOperator{\Ext}{Ext}

\newcommand{\isom}{\cong}
\newcommand{\normaleq}{\unlhd}
\DeclareMathOperator{\iso}{iso}
\newcommand{\leqc}{\leq_c}
\newcommand{\dirlim}{\varinjlim}
\newcommand{\invlim}{\varprojlim}

\DeclarePairedDelimiter{\abs}{\lvert}{\rvert}
\DeclarePairedDelimiter{\gen}{\langle}{\rangle}

\DeclarePairedDelimiter{\dbrack}{[\mkern-1.5mu[}{]\mkern-1.5mu]}

\DeclarePairedDelimiter{\dangle}{\langle\mkern-2.75mu\langle}{\rangle\mkern-2.75mu\rangle}

\theoremstyle{plain}

\newtheorem{lemma}{Lemma}[section]
\newtheorem{proposition}[lemma]{Proposition}
\newtheorem{theorem}[lemma]{Theorem}
\newtheorem{corollary}[lemma]{Corollary}

\theoremstyle{definition}

\newtheorem{definition}[lemma]{Definition}
\newtheorem{remark}[lemma]{Remark}

\newtheorem{example}[lemma]{Example}

\newtheorem*{thm}{Theorem}
\newtheorem*{Conjecture}{Conjecture}

\newtheorem*{remark*}{Comment}

\title{Cohomology of solvable saturable pro-$p$ groups and Lie algebras}

\author{Oihana Garaialde Oca\~na, Jon Gonz\'alez-S\'anchez, Lander Guerrero-S\'anchez
 \thanks{ 
    \noindent The  authors were partially supported by the Spanish Government grant  PID2020-117281GB-I00 and PCI2024-155096-2, partly
with FEDER funds, and by the Basque Government project IT974-16.}}

\begin{document}

\maketitle
	
\begin{abstract} Let $p$ be an odd prime, and let $n\in \N$ be an integer. We show that the $n^{\text{th}}$ mod-$p$ cohomology of a solvable saturable pro-$p$ group is isomorphic to the $n^{\text{th}}$ mod-$p$ cohomology of its associated $\Z_p$-Lie algebra $\g$ as a $\F_p$-vector space. Addittonally, we obtain that the $n^{\text{th}}$ mod-$p$ cohomology of $\g$ and of $\g/p\g$ are isomorphic as $\F_p$-vector spaces.
\end{abstract}

\section{Introduction}

In the celebrated paper  \cite{lazard65}, Lazard studied `$p$-adic analytic groups', which roughly speaking are topological groups that also have an analytic structure over the $p$-adic numbers $\Qp$. Lazard also gave a more algebraic flavored characterization of the previous notion: a topological group is $p$-adic analytic if and only if it contains an open saturable pro-$p$ subgroup (see Section \ref{subsec: Groups}). Moreover, Lazard shows that there is a (saturable) $\Z_p$-Lie algebra associated to every saturable pro-$p$ group. Even more, he establishes an isomorphism of categories between the category of saturable pro-$p$ groups and the category of saturable $\Z_p$-Lie algebras. 

The analytic essence of such groups has been hidden by their algebraic reinterpretation given by Lubotzky and Mann in \cite{LMpowerful}. Indeed, they introduce the concept of `powerful' groups, and show that a topological group is $p$-adic anaytic if and only if it contains an open powerful pro-$p$ subgroup. Here, for $p$ odd, a pro-$p$ group $G$ is powerful if $[G,G]\leq G^p$ holds. Equivalently, they show that $G$ is a $p$-adic analytic group if and only if $G$ has finite rank  \cite{LMpowerfulII}, and if $p$ is odd and $G$ is  torsion-free, this value coincides with its dimension $\dim G$ as a manifold over $\Q_p$.

Saturable groups are closely related to `uniformly powerful' pro-$p$ groups—often referred to simply as `uniform' groups—which are finitely generated, torsion-free, powerful pro-$p$ groups. In fact, every uniform group is saturable; see \cite{LMpowerfulII}, \cite{FGJ}. However, these two classes of groups do not coincide; see \cite{klopsch05}. Additional related notions appear, for instance, in \cite{gonzalez07}, \cite{gonzalezjaikin}.

Among the results presented in his seminal work, Lazard studies the relation between the cohomology of compact $p$-adic analytic groups and of their associated Lie algebras. We present a weaker version of his result: let $G$ be a compact $p$-adic analytic group,  and let $\g$ be its associated $\Q_p$-Lie algebra. Then, there is an isomorphism of graded-commutative $\Q_p$-algebras
\[
\HH^*(G;\Q_p)\cong \HH^*(\g;\Q_p)^G
\]
between the cohomology of $G$ and the $G$-stable elements of the cohomology of $\g$ with coefficients in $\Q_p$. Additionally, Lazard proves that the cohomology of `\'equi-$p$ valu\'e'  groups and of their associated $\Z_p$-Lie algebras $\g$ with coefficients in $\F_p$ are isomorphic to the exterior algebra on $\g/p\g$. Specifically, if $G$ is an \'equi-$p$-valu\'e group with associated $\Z_p$-Lie algebra $\g$, then there are isomorphisms
\[
\HH^*(G;\F_p)\cong \HH^*(\g;\F_p)\cong\Lambda^*(\g/p\g),
\]
where the right-hand side denotes the exterior algebra on $\g / p\g$. For the $p$ odd case, \'equi-$p$ valu\'e groups correspond to uniform pro-$p$ groups.

In this manuscript, we aim to extend this isomorphism to the broader class of saturable pro-$p$ groups. Accordingly, we formulate the following conjecture:

\begin{Conjecture}
	Let $G$ be a saturable pro-$p$ group with associated $\Zp$-Lie algebra $\g$. Then, there is an isomorphism of $\Fp$-algebras 
	\begin{displaymath}
		\HH^{*}(G;\Fp)\isom \HH^{*}(\g;\Fp)\cong \HH^*(\g/p\g;\F_p).
	\end{displaymath}
\end{Conjecture}

It is not difficult to see that the Heisenberg group and its associated Heisenberg Lie algebra over $\Z_p$ satisfy the above conjecture. As a first step towards the proof of the above conjecture we focus on the solvable case.

\begin{thm}\label{thm: main}
Let $p$ be an odd prime number, let $G$ be a solvable saturable pro-$p$ group and let $\g$ denote its associated $\Z_p$-Lie algebra. For every integer $n\geq 0$, there is a isomorphism of $\F_p$-vector spaces
\[
\HH^n(G;\F_p)\cong \HH^n(\g;\F_p)\cong \HH^n(\g/p\g;\F_p).
\]
\end{thm}

\bigskip

\noindent
{\bf Notation and organization.} Throughout, $p$ denotes an odd prime number, and $\Z_p$ denotes the ring of $p$-adic integers. Unless otherwise stated, cohomology of a group $G$ or of a Lie algebra $\g$ is assumed to be with coefficients in the finite field $\F_p$ of $p$ elements; denoted by $\HH^*(G)$ and by $\HH^*(\g)$, respectively. 

In Section \ref{sec: LazardCorr}, we set a base on saturable pro-$p$ groups and saturable $\Z_p$-Lie algebras, and we present Lazard's correspondence. In Section \ref{sec: cohomology}, we provide basic notions on cohomology of pro-$p$ groups and of $\Z_p$-Lie algebras. Section \ref{sec: solvable}  is devoted to prove the above Theorem, and in the last section we present some examples and propose directions for further work.

\section{Lazard Correspondence: saturable pro-$p$ groups and saturable Lie algebras}\label{sec: LazardCorr}

 As we mentioned before, in his seminal work \cite{lazard65}, Lazard introduced the concept of saturable pro-$p$ groups and showed a correspondence between this class of groups and that of saturable $\Zp$-Lie algebras. The aim of this chapter is to introduce the aforementioned concepts and to briefly explain the correspondence. Throughout, we assume that $p$ is an odd prime number.
	
	\subsection{Groups}\label{subsec: Groups}
	
		 We start by giving the definition of $PF$ and saturable groups, and by illustrating their connection with $p$-adic analytic groups.
	
	\begin{definition}
		Let $G$ be a pro-$p$ group and let $N\leq G$ be a subgroup. We say that $N$ is \emph{PF-embedded} in $G$ if there is a family $\{N_i\}_{i\in \N}$ of closed normal subgroups of $G$ satisfying the following properties:
		\begin{enumerate}
			\item[{\bf i.}] $N_1=N$.
			
			\item[{\bf ii.}] $N_{i+1}\leq N_i$ for all $i\in \N$.
			
			\item[{\bf iii.}] $\bigcap_{i\in \N}N_i=1$.
			
			\item[{\bf iv.}] $[N_i,G]\leq N_{i+1}$ for all $i\in \N$.
			
			\item[{\bf v.}] $[N_i,_{p-1}G]\leq N_{i+1}^p$ for all $i\in \N$. 
		\end{enumerate}
		We say that $G$ is a \emph{PF-group} if it is PF-embedded in itself.
	\end{definition}
	
	Observe that if $G$ is a $PF$-group, then $G^p$ is an open powerful subgroup of $G$ (see \cite[Proposition 3.2, Corollary 3.5]{FGJ}) and thus, we have the following result. 
	\begin{proposition}
		Every PF-group is $p$-adic analytic.
	\end{proposition}

	\begin{definition}
		A finitely generated pro-$p$ group is said to be \emph{saturable} if it is a torsion-free PF-group.
	\end{definition}
	
		\begin{remark}
		This definition is equivalent to the original one given by Lazard in \cite{lazard65} using valuations, see \cite[Theorem 3.4]{gonzalez07}.
	\end{remark}
	
	Beside uniform pro-$p$ groups, examples of saturable pro-$p$ groups are those coming from the class of torsion-free compact $p$-adic analytic pro-$p$ groups of dimension smaller than $p$ (see \cite{gonzalezklopsch09}), and finitely generated torsion-free nilpotent pro-$p$ groups of nilpotency class smaller than $p$. The concept of saturable pro-$p$ group allows us to give another characterisation of $p$-adic analytic groups \cite{lazard65}.

\begin{theorem}
	A pro-$p$ group is $p$-adic analytic if and only if it is finitely generated and virtually saturable. 
\end{theorem}
	
	We finish this section by defining an important concept that will be necessary in the proof of our result (see Section \ref{sec: solvable}).
	
	\begin{definition}
		Let $G$ be a pro-$p$ group and $H\leqc G$. The \emph{isolator} of $H$ in $G$ is the closed subgroup
		\begin{displaymath}
			\iso_G(H)=\gen{g\in G\mid g^{p^k}\in H \text{ for some } k\in \N}.
		\end{displaymath}
	\end{definition}
	
	\begin{proposition}[{\cite[Proposition 3.21]{gonzalezklopsch09}}]\label{prop: isolator}
		Let $G$ be a saturable pro-$p$ group and $H\leqc G$. Then, $\iso_G(H)$ is saturable with  $\abs{\iso_G(H):H}<\infty$, and 
		\begin{displaymath}
			\iso_G(H)=\{g\in G\mid g^{p^k}\in H \text{ for some } k\in \N\}.
		\end{displaymath}
	Furthermore, if $H\normaleq G$ then $\iso_G(H)$ is PF-embedded in $G$. 
	\end{proposition}

	\subsection{Lie algebras}
	
	To establish Lazard's correspondence between groups and Lie algebras, it is necessary to introduce the Lie-theoretic analogues of the corresponding group-theoretic concepts.

	\begin{definition}
		Let $\g$ be a finitely generated $\Zp$-Lie algebra and let $\n\subseteq \g$ be a Lie subalgebra. We say that $\n$ is \emph{PF-embedded} in $\g$ if there is a family $\{\n_i\}_{i\in \N}$ of Lie ideals of $\g$ satisfying the following properties:
		\begin{enumerate}
			\item[{\bf i.}] $\n_1=\n$.
			
			\item[{\bf ii.}] $\n_{i+1}\subseteq \n_i$ for all $i\in \N$.
			
			\item[{\bf iii.}] $\bigcap_{i\in \N}\n_i=0$.
			
			\item[{\bf iv.}] $[\n_i,\g]\leq \n_{i+1}$ for all $i\in \N$.
			
			\item[{\bf v.}] $[\n_i,_{p-1}\g]\leq p\n_{i+1}$ for all $i\in \N$. 
		\end{enumerate}
		We say that $\g$ is a \emph{PF-Lie algebra} if it is PF-embedded in itself.
	\end{definition}
	
	\begin{definition}
		Let $\g$ be a finitely generated $\Zp$-Lie algebra. We say that $\g$ is a \emph{saturable $\Zp$-Lie algebra} if it is a PF-Lie algebra which is free as a $\Zp$-module. 
	\end{definition}

	Once again, uniform $\Zp$-Lie algebras, i.e. powerful and free as a $\Zp$-modules, constitute the classical examples of saturable $\Zp$-Lie algebras.

		\subsection{The correspondence}

	Let $G$ be a saturable pro-$p$ group. Then, we can construct its associated saturable $\Zp$-Lie algebra $\g=\log(G)$ with underlying set $G$ as follows. Given $x,y\in G$ and $\lambda\in \Zp$, define 
	\begin{displaymath}
		\lambda x=x^{\lambda}, \qquad x+y =\lim_{n\to \infty} (x^{p^n}y^{p^n})^{p^{-n}}, \qquad [x,y]=\lim_{n\to \infty}[x^{p^n},y^{p^n}]^{p^{-2n}}. 
	\end{displaymath}
	Conversely, given a saturable $\Zp$-Lie algebra $\g$, we can construct a saturable pro-$p$ group $G=\exp(\g)$ with underlying set $\g$ and product given as follows: for $x,y\in\g$, define 
	\begin{displaymath}
		xy=\Phi(x,y),
	\end{displaymath}
	where $\Phi(X,Y)\in\Q\dangle{X,Y}$ is the Baker-Campbell-Hausdorff formula, which can be written as a formal series
	\begin{displaymath}
		\Phi(X,Y)=X+Y + \sum_{k=1}^{\infty}u_k(X,Y) =\log\big(\exp (X)\exp (Y)\big)
	\end{displaymath}
	with $u_k(X,Y)\in \Q\langle X,Y\rangle$ a Lie polynomial of degree $k\geq 1$, see \cite[Section 6.5]{dixondusautoy99} and \cite[Section VII.32]{schneider13}. Here, $\Q\langle X,Y\rangle$ and $\Q\dangle{X,Y}$ denote the $\Q$-algebras of non-commutative polynomials and power series in the indeterminates $X,Y$, respectively.

	This gives rise to a correspondence between the categories of saturable pro-$p$ groups and saturable $\Zp$-Lie algebras via $\log$ and $\exp$ functors. Moreover, this correspondence sends PF-embedded subgroups to PF-embedded subalgebras, and vice-versa, see \cite{gonzalezklopsch09}.

\section{Cohomology of pro-$p$ groups and of $\Z_p$-Lie algebras}\label{sec: cohomology}

\subsection{Groups}

	The aim of this subsection is to set a base on cohomology of (finitely generated) pro-$p$ groups that will be used in the proof of the main result; see \cite[Chapter 6]{ribeszalesskii13} for further information. Let $G$ be a pro-$p$ group. The completed group algebra of $G$ over $\Zp$ is defined by the inverse limit over the open normal subgroups $N$ of $G$:
	$$
	\Zp\dbrack{G}=\invlim_{N\normaleq_o G}\Zp[G/N].
	$$ 
	For every $n\geq 0$, the continuous $n^{\text{th}}$ cohomology group of $G$ with coefficients on a discrete $\Zp\dbrack{G}$-module $V$ is defined by
\begin{displaymath}
	\HH^n(G,V)=\Ext_{\Zp\dbrack{G}}^n(\Zp,V),
\end{displaymath}
and can be computed as the direct limit of the (discrete) $n^{\text{th}}$ cohomology groups: 
\begin{displaymath}
	\HH^n(G,V)=\dirlim_{N\normaleq_o G}\HH^n(G/N,V^N)=\dirlim_{N\normaleq_o G}\Ext_{\Zp[G/N]}^n (\Zp,V^N).
\end{displaymath}

\begin{example}\label{ex: GroupDimOne}
	Let us compute the cohomology of $G=\Zp$ with coefficients on a $\Zp\dbrack{G}$-module $V$. If we write $G=\gen{\sigma}$, then $\Zp\dbrack{G}=\Zp\dbrack{\sigma-1}$. It is easy to compute that 
	\begin{displaymath}
		\HH^r(G,V)\isom \begin{dcases}
			V^{G} & \text{if } r=0, \\
			\frac{V}{(\sigma-1)V} & \text{if } r=1, \\
			0 & \text{if } r\geq 2. \\
		\end{dcases}
	\end{displaymath}
\end{example}

Just as for discrete cohomology, whenever we have an extension of pro-$p$ groups there is an associated Lyndon-Hochschild-Serre spectral sequence.
\begin{theorem}[{\cite[Theorem 7.2.4]{ribeszalesskii13}}]
	Given an extension of pro-$p$ groups of the form
	\[\begin{tikzcd}
		1 \rar & K \rar & G \rar & G/K \rar &  1,
	\end{tikzcd}\]
	 there is a cohomological type spectral sequence $E$ with 
	\begin{displaymath}
		E_2^{r,s}=\HH^r\big(G/K,\HH^s(K,V)\big) \implies \HH^{r+s}(G,V).
	\end{displaymath}
\end{theorem}

\subsection{Lie algebras and Eckmann-Shapiro lemma}

We set a base on cohomology of Lie algebras; see \cite[Chapter 7]{weibel94} for further information. Let $\g$ be a $\Zp$-Lie algebra and consider its universal enveloping algebra $\U\g$. For an integer number $n\geq 0$, the $n^{\text{th}}$ cohomology of $\g$ with coefficients on a $\U\g$-module $V$ is defined by
\begin{displaymath}
	\HH^n(\g;V)=\Ext_{\U\g}^n(\Zp,V).
\end{displaymath}
If $\g$ is a finitely generated free $\Zp$-module (i.e., of finite rank), this cohomology can be computed via the Chevalley-Eilenberg complex: for every $n\geq 0$, define 
\begin{displaymath}
	W_n\g=\U\g\otimes \Lambda_n\g
\end{displaymath}
with differential $\dd_{n+1}\colon W_{n+1}\g\longrightarrow W_{n}\g$ given by
\begin{align*}
	\dd_{n+1}(u\otimes x_1\dotsm x_{n+1})&=\sum_{i=1}^{n+1}(-1)^{i+1}ux_i\otimes x_1\dotsm\hat x_i\dotsm x_{n+1} \\
	& \qquad +
	\sum_{1\leq i<j\leq n+1}(-1)^{i+j} u\otimes [x_i,x_j]x_1\dotsm\hat x_i \dotsm\hat x_j\dotsm x_{n+1}.
\end{align*}
 This complex gives a free $\U\g$-resolution $W_{\bullet}\g\longrightarrow \Zp$ of the trivial module, see \cite[Section 7.7]{weibel94}.

Once again, there is a Lyndon-Hochschild-Serre spectral sequence for Lie algebra extensions.
\begin{theorem}[{\cite[Section 7.5]{weibel94}}]
	Given an extension of $\Zp$-Lie algebras of the form
	\[\begin{tikzcd}
		0 \rar & \kk \rar & \g \rar & \g/\kk \rar &  0,
	\end{tikzcd}\]
	and a $\U\g$-module $V$, there is a cohomological type spectral sequence $E$ with 
	\begin{displaymath}
		E_2^{r,s}=\HH^r\big(\g/\kk;\HH^s(\kk;V)\big) \implies \HH^{r+s}(\g;V).
	\end{displaymath}
\end{theorem}

The following result is a variant of the Eckmann-Shapiro Lemma with weaker assumptions; see \cite[page 118]{CE1956}.

\begin{lemma}[Eckmann-Shapiro]\label{lem: eckmann-shapiro}
	Let $R\longrightarrow S$ be a ring homomorphism that makes $S$ a $R$-module, $V$ be a left  $S$-module and $W$ be a left $R$-module. Assume that there is a projective $R$-resolution $ P_{\bullet} \longrightarrow W$ of $W$ such that $S\otimes_RP_{\bullet}\longrightarrow S\otimes_R W$ is a projective $S$-resolution of $S\otimes_R W$. Then, for each $n\geq 0$ there is an isomorphism 
	\begin{displaymath}
		\Ext_R^n (W,V_R)\isom \Ext_S^n(S\otimes_R W, V).
	\end{displaymath}
\end{lemma}

The assumption that  there is a projective $R$-resolution $ P_{\bullet} \longrightarrow W$ of $W$ such that $S\otimes_RP_{\bullet}\longrightarrow S\otimes_R W$ is a projective $S$-resolution of $S\otimes_R W$ of the previous statement holds in the particular case where $S$ is a flat $R$-module.

\begin{remark}
	If $\g$ is a $\Zp$-Lie algebra, then $\g/p\g$ is a $\Fp$-Lie algebra. Given a $(\g/p\g)$-module $V$, using Lemma \ref{lem: eckmann-shapiro} with the Chevalley-Eilenberg complex, we can show that 
	\begin{displaymath}
		\HH^{*}(\g;V)=\Ext_{\U\g}^{*}(\Zp,V)\isom\Ext_{\U(\g/p\g)}^{*}(\Fp,V)=\HH^{*}(\g/p\g;V).
	\end{displaymath}
\end{remark}

\begin{example}\label{ex: LieDimOne}
	Let us compute the cohomology of $\g=\Zp$ with coefficients on a $\U\g$-module $V$. If we write $\g=\gen{S}$, then $\U\g=\Zp[S]$. It is easy to compute that 
	\begin{displaymath}
		\HH^r(\g;V)\isom \begin{dcases}
			V^{\g} & \text{if } r=0, \\
			\frac{V}{SV} & \text{if } r=1, \\
			0 & \text{if } r\geq 2. \\
		\end{dcases}
	\end{displaymath}
\end{example}

	\section{The solvable case}\label{sec: solvable}

	Our main goal is to compare the mod-$p$ cohomology of $G$ with that of $\g$, and in doing so, show that they are isomorphic. This seems a challenging problem, and to start with, we focus on the case when $G$ is solvable. We need the following two results.

	\begin{lemma}\label{lem: abelianisation infinite}
		If $G$ is a solvable saturable pro-$p$ group, then $G/G'$ is infinite.
	\end{lemma}

	\begin{proof}
		We will prove the result by induction on the derived length $n$ of $G$. If $n=1$, then $G$ is abelian and $G/G'=G$ is torsion-free, thus infinite.
		
		Assume that the result holds for groups of derived length $n-1$. This implies that $G'/G''$ is infinite. Suppose by contradiction that $G/G'$ is finite. Then, there is some $k\in \N$ such that, for all $g\in G$, $g^{p^k}\in G'$ holds; implying that $G^{p^k}\leq [G,G]$. Therefore, we must have that
		\begin{displaymath}
			[G,G]^{p^{2k}}=[G^{p^k},G^{p^k}]\leq G'',
		\end{displaymath}
		contradicting the fact that $G'/G''$ is infinite.
	\end{proof}

	\begin{proposition}\label{prop: subgroup series}
		Let $G$ be a solvable saturable pro-$p$ group of dimension $\dim G=n$. Then, there is a series of closed subgroups 
		\begin{displaymath}
			G=K_0> K_1 > \dotsb > K_n=1,
		\end{displaymath}
		such that:
		\begin{enumerate}
			\item[{\bf i.}] $K_{i+1}\normaleq K_i$ for all $i=0,\dotsc,n-1$.
			
			\item[{\bf ii.}] $K_i/K_{i+1}\isom \Zp$ for all $i=0,\dotsc,n-1$.
			
			\item[{\bf iii.}] $K_i$ is saturable for all $i=0,\dotsc,n$.
			
			\item[{\bf iv.}] $K_{i+1}$ is PF-embedded in $K_i$ for all $i=0,\dotsc,n-1$.
		\end{enumerate}
	\end{proposition}
	\begin{proof}
		Denote by $T\leq G/G'$ the torsion subgroup of $G/G'$. By Lemma \ref{lem: abelianisation infinite}, $G/G'$ is infinite, and so there is some $s\geq 1$ such that 
		\begin{displaymath}
			G/G'\isom \Zp^s \times T.
		\end{displaymath}
		We can now a take a closed normal subgroup $G'\normaleq K_1\normaleq G$ such that $G/K_1\isom \Zp$. Furthermore, 
		 it is easy to see that $K_1=\iso_G(K_1)$. Indeed, $K_1\leq \iso_G(K_1)$ and, given $g\in \iso_G(K_1)$ with $g^{p^k}\in K_1$, the fact that $G/K_1$ is torsion-free implies that $g\in K_1$. Consequently, $K_1$ is saturable and PF-embedded in $G$ by Proposition \ref{prop: isolator}. Moreover, we have that 
		\begin{displaymath}
			\dim K_1 = \dim G -1.
		\end{displaymath}
	Thus, we can inductively obtain a series
	\begin{displaymath}
		G=K_0> K_1 > \dotsb > K_n=1
	\end{displaymath}
	satisfying the required properties.
	\end{proof}

	Until the end of this section, we proceed to prove the main result of the paper. Let $G$ be a  solvable saturable pro-$p$ group. We would like to show, by induction on the dimension $\dim G$ that, $\HH^{n}(G)\isom \HH^{n}(\g)$ for all $n\geq 0$. The base case, that is, the case where $G\cong \Z_p$ has dimension one holds by Examples \ref{ex: GroupDimOne} and \ref{ex: LieDimOne}. 
	
	Let $n\in \N_{>1}$ be a natural number, and set $n=\dim G$. Assume now that the result follows for all solvable saturable pro-$p$ groups of dimension smaller than $n$. 
	By Proposition \refeq{prop: subgroup series}, there is some PF-embedded (saturable) normal subgroup $K\normaleq G$ such that $G/K\isom \Zp$.  Furthermore, we have that $\dim K= \dim G -1$.  Let $\kk$ denote the $\Z_p$-Lie algebra associated to $K$. Then, by induction  hypothesis, there is an ($G$-equivariant) isomorphism $$\HH^n(K)\isom\HH^n(\kk)$$
	of $\F_p$-vector spaces, for all $n\geq 0$. Now, the strategy is to consider the LHS spectral sequence $E_2$ associated to the group extension
	\[\begin{tikzcd}
		1 \rar & K \rar & G \rar & G/K \rar &  1.
	\end{tikzcd}\]
	Then, we have that 
	\begin{displaymath}
		E_2^{r,s}(G)=\HH^r\big(G/K;\HH^s(K)\big)\isom \HH^r\big(\Z_p;\HH^s(K)\big).
	\end{displaymath}
	Similarly, using Lazard's correspondence, there is an extension of the corresponding $\Z_p$-Lie algebras
		\[\begin{tikzcd}
		0 \rar & \kk \rar & \g \rar & \g/\kk \rar &  0,
	\end{tikzcd}\]
	from which we can compute the second page $E_2^{r,s}(\g)$ of the LHS spectral sequence. The objective is to show that these two spectral sequences, $E_2(G)$ and $E_2(\g)$, are isomorphic as bigraded $\F_p$-vector spaces. To that aim, we need to prove the following result.

		\begin{lemma}\label{lem: abelian} Let $Q\cong \Z_p$ and let $\q\cong \Z_p$ be its associated $\Z_p$-Lie algebra. Let $V$ be an $\F_p$-vector space such that the action of $Q$ on $V$ is unipotent of class at most $p$. Then, for all $n\geq 0$, we have an $Q$-equivariant isomorphism 
		\begin{displaymath}
			\HH^n(Q;V) \isom \HH^n(\q; V)
		\end{displaymath}
		of  $\q$-modules, where the action of $\q$ on $\HH^n(\q; V)$ is given by the $\log$ functor.
	\end{lemma} 
	
	\begin{proof}
	Let $Q\cong\Zp=\gen{\sigma}$ and $\q\cong\Zp=\gen{T}$ be its associated $\Zp$-Lie algebra, with $\sigma=\exp(T)$ and $T=\log_p(\sigma)$. Then, we have that 
\[
		\Zp\dbrack{Q}=\Zp\dbrack{\sigma-1} \quad \text{and} \quad \U\q =\Zp[T].
\]
	Thus, for each $n\geq 0$, the $n^{\text{th}}$ cohomology of $G$ and of $\g$, with coefficients in a $\Zp\dbrack{Q}$-module $V$ and in a $\U\q$-module $U$ are
$$
	\HH^n(Q;V)=\Ext_{\Zp\dbrack{Q}}^n(\Zp,V) \quad \text{and}\quad 
	\HH^n(\q;U)=\Ext_{\U\q}^n(\Zp,U),
$$
	respectively. Consider the ring homomorphism $\Psi\colon \Zp[T]\longrightarrow \Zp\dbrack{\sigma-1}$ defined by using the truncated logarithm 
	\begin{displaymath}
		\Psi(T)=\log_p \sigma=\sum_{k=1}^{p-1}\frac{(-1)^{k+1}}{k}(\sigma-1)^k.
	\end{displaymath}
	Given a $\Zp\dbrack{\sigma-1}$-module $V$, denote by $\log_p V$ the $\Zp[T]$-module with underlying set $V$ and $\Zp[T]$-action induced by $\Psi$. 
	The homomorphism $\Psi$ turns $\Zp\dbrack{\sigma-1}$ into a flat $\Zp[T]$-module, and so we can apply Lemma \ref{lem: eckmann-shapiro} to obtain an isomorphism
	\begin{displaymath}
		\Ext_{\Zp[T]}^n (\Zp,\log_p V)\isom \Ext_{\Zp\dbrack{\sigma-1}}^n\big(\Zp\dbrack{\sigma-1}\otimes_{\Zp[T]} \Zp, V\big).
	\end{displaymath}
	Furthermore, we can easily show that $$\Zp\dbrack{\sigma-1}\otimes_{\Zp[T]} \Zp\isom
	\Zp,$$
	from which we obtain the result.
	\end{proof}

	We now proceed to state and to prove the main result of this manuscript. Let $G$ be a solvable saturable pro-$p$ group and let $\g$ be its associated $\Z_p$-Lie algebra. Let $K\normaleq G$ be the normal subgroup described in Proposition \ref{prop: subgroup series}. Using Lazard's correspondence, let $\kk=\log K$ be its associated $\Z_p$-Lie algebra, and notice that, $\g/\kk=\log (G/K)\cong \Z_p$. Consider the extensions
	\[
	1\to K\to G\to G/K \to 1, \quad \text{and}\quad 0\to \kk \to \g\to \g/\kk\to 0,
	\]
and their associated spectral sequences 
$$
E^{r,s}_2(G)\cong \HH^r\big(G/K;\HH^s(K,\Fp)\big) \; \text{ and  } \; E^{r,s}_2(\g)\cong  \HH^r\big(\g/\kk;\HH^s(\kk,\Fp)\big),
$$
respectively.

	
	\begin{theorem}  
		Under the hypothesis described above, for all $r,s\geq 0$, we have that $E^{r,s}_2(G)\cong E_2^{r,s}(\g)$, and therefore, $E^{r,s}_{\infty}(G)\cong E^{r,s}_{\infty}(\g)$ as $G/K$-modules.

	\end{theorem}
	
	\begin{proof}
	
	Assume by induction that, for all $n\geq 0$, there is a $G$-equivariant isomorphism $$\HH^n(K;\Fp)\isom\HH^n(\kk;\Fp).$$
	Then, we have that 
	\begin{displaymath}
		E_2^{r,s}(G)=\HH^r\big(G/K;\HH^s(K,\Fp)\big)\isom \begin{dcases}
			\HH^s(K;\Fp)^{G/K} & \text{if } r=0, \\
			\frac{\HH^s(K;\Fp)}{(\sigma-1)\HH^s(K;\Fp)} & \text{if } r=1, \\
			0 & \text{if } r\geq 2, \\
		\end{dcases}
	\end{displaymath}
	and, similarly,
	\begin{displaymath}
		E_2^{r,s}(\g)=\HH^r\big(\g/\kk,\HH^s(\kk,\Fp)\big)\isom \begin{dcases}
			\HH^s(\kk,\Fp)^{\g/\kk} & \text{if } r=0, \\
			\frac{\HH^s(\kk,\Fp)}{T\HH^s(\kk,\Fp)} & \text{if } r=1, \\
			0 & \text{if } r\geq 2. \\
		\end{dcases}
	\end{displaymath}
To show that for all $r,s\geq 0$, Lazard's correspondence gives an isomorphism of $\F_p$-vector spaces $E^{r,s}_2(G)\cong E^{r,s}_2(\g)$, we set $V=\HH^s(K,\Fp)\isom \HH^s(\kk,\Fp)$ and verify that the conditions of Lemma \ref{lem: abelian} are satisfied. Indeed, $G$ acts on $\kk$ and thus, on $\kk/p\kk$, by the $\exp$ of the adjoint. Moreover, since $G$ is saturable, we have that $\gamma_p(G)\leq G^p$ and  thus, the action of $G$ on $\kk/p\kk$ is unipotent of class at most $p$. More explicitly, the generator $T=\log \sigma\in \q=\g/\kk$ acts on $x\in \kk$ via
	\begin{displaymath}
		x\cdot T=\sum_{i=1}^{p-1}\frac{(-1)^{i+1}}{i}x(C_\sigma-\id_Q)^i =x\cdot\log(C_{\sigma}),
	\end{displaymath}
	where $xC_\sigma=x^\sigma$ is conjugation by $\sigma\in Q$. Hence, $x(C_\sigma-\id)^i\in p\kk$ for all $x\in \kk$ and $i\geq p$. Hence, the generator $T\in\q$ acts on the $\Fp$-module $\Hom(\Lambda^s \kk,\Fp)$, and thus on $\HH^s(\kk,\Fp)$, as $\log \sigma$. In conclusion, the action of $\q$ on $\HH^*(\kk;\F_p)$ is unipotent of class at most $p$ and therefore, we can deduce from Lemma \ref{lem: abelian}, that
	
	\begin{displaymath}
		\HH^r\big(Q;\HH^s(K)\big) \isom \HH^r\big(\q;\HH^s(\kk)\big)
	\end{displaymath}
	as $G/K$-modules. Finally, observe that all the differentials of both spectral sequences are trivial, and hence, $E_{\infty}=E_2$ holds in both cases.
	\end{proof}

	As we already mentioned, $E_2\cong E_{\infty}$ holds for both spectral sequences and we obtain the following result.
	
	\begin{corollary}
		For any $n\geq 0$, we have an isomorphism
		\begin{displaymath}
			\HH^n(G)\isom \HH^n(\g)
		\end{displaymath}
		as $\Fp$-modules.
	\end{corollary}

\section{Examples and further work}

Let $p>3$ be a prime number. Let $G$ be a torsion-free solvable pro-$p$ group such that $G/G^p$ is a generalized extraespecial finite $p$-group. Then, $G$ is saturable and its associated $\Z_p$-Lie algebra $\g$ satisfies that $\g/p\g$ is a generalized Heisenberg Lie algebra over $\F_p$. Since the Betti numbers of $\g/p\g$ have been computed in \cite{CaJa08}, applying our result, the Betti numbers of $G$ are known. For example, the following pro-$p$ groups fall into the aforementioned family: for $n\in \N$,
\begin{align*}
G_n=&\langle x_1, \dots, x_n, y_1, \dots, y_n, z \; \mid \; [x_i, x_j]=[y_i, y_j]=[x_i, y_j]=[x_i, z]= [y_j,z]= 1,\\
& [x_i,y_i]=z, \, \text{for all}\, i\neq j\in \{1, \dots,m\}\rangle.
\end{align*}

It would be interesting to provide further explicit examples of this flavor. Furthermore, obtaining concrete computations of the cohomology of saturable groups and of their associated Lie algebras would be highly desirable, as these could serve as meaningful tests for our conjecture.

A different direction would be to compare the cohomology of a saturable group over  $\F_p$ and over $\Q_p$. While the structures of the respective cohomology algebras can differ significantly, one might expect that, under suitable conditions, both cohomology algebras could be realized by the same representable functor.

\begin{Conjecture}
	Let $G$ be a saturable nilpotent pro-$p$ group with associated $\Zp$-Lie algebra $\g$.  Assume that the number of generators of $G$ as a pro-$p$ group 
	and of $\g\otimes_{\Z_p} \Q_p$ as a $\Q_p$-Lie algebra coincide. Then there exists a representable functor of commutative graded $\Z_p$-algebras $A$ such that
	\begin{displaymath}
		\HH^{*}(G;\Fp)\cong A(\F_p) \text{\  and\ } \HH^{*}(G;\Qp)\cong A(\Q_p) .
	\end{displaymath}
\end{Conjecture}

Observe that this conjecture holds for the Heissenberg group over $\Z_p$. Indeed, let $G$ denote the Heisenberg group over $\Z_p$ then both cohomology algebras $\HH^*(G)$ and $\HH^*(G;\Q_p)$ are represented by the exterior algebra $\Lambda(x,y,X,Y,Z)$ on five generators of degrees $ |x|=|y|=1$, $|X|=|Y|=2, |Z|=3$ and satisfying the following relations
\[
xy=xX=yY=XY=0, xY=yX=Z.
\]

Based on the previous example together with the technics used through the proofs it is natural to extend the previous conjecture to the following one:

\begin{Conjecture}
	Let $G$  be a nilpotent affine group scheme over the integers. Then there exist a representable functor of commutative graded $\Z$-algebras $A$ such that for almost all primes $p$
	\begin{displaymath}
		\HH^{*}(G(\Z_p);\Fp)\cong A(\F_p) \text{\  and\ } \HH^{*}(G(\Z_p);\Qp)\cong A(\Q_p) .
	\end{displaymath}
\end{Conjecture}

\bibliographystyle{plain}

\bibliography{bib}

\begin{thebibliography}{10}

\bibitem{CaJa08}
G.~{Cairns} and S.~{Jambor}.
\newblock The cohomology of the heisenberg lie algebras over fields of finite
  characteristic.
\newblock {\em Proceedings of the American Society}, 126:3803--3807, 2008.

\bibitem{CE1956}
H.~Cartan and S.~Eilenberg.
\newblock {\em Homological Algebra}.
\newblock Princeton University Press, 1st edition edition, 1956.

\bibitem{dixondusautoy99}
J.~D. Dixon, M.~P.~F. du~Sautoy, A.~Mann, and D.~Segal.
\newblock {\em Analytic pro-$p$ Groups}.
\newblock Cambridge Studies in Advanced Mathematics. Cambridge University
  Press, 2nd edition, 1999.

\bibitem{FGJ}
G.~{Fernández-Alcober}, J.~{Gonz\'alez-S\'anchez}, and A.~{Jaikin-Zapirain}.
\newblock Omega subgroups of pro-$p$ groups.
\newblock {\em Israel Journal of Mathematics}, 166:393--412, 2008.

\bibitem{gonzalez07}
J.~{Gonz\'alez-S\'anchez}.
\newblock On $p$-saturable groups.
\newblock {\em Journal of Algebra}, 315(2):809--823, 2007.

\bibitem{gonzalezjaikin}
J.~{Gonz\'alez-S\'anchez} and A.~{Jaikin-Zapirain}.
\newblock Analytic pro-$p$ groups of small dimensions.
\newblock {\em Journal of Algebra}, 276(1):193--209, 2004.

\bibitem{gonzalezklopsch09}
J.~{Gonz\'alez-S\'anchez} and B.~Klopsch.
\newblock Analytic pro-$p$ groups of small dimensions.
\newblock {\em Journal of Group Theory}, 12(5):711--734, 2009.

\bibitem{klopsch05}
B.~Klopsch.
\newblock On the {Lie} theory of $p$-adic analytic groups.
\newblock {\em Mathematische Zeitschrift}, 249(4):713--730, April 2005.

\bibitem{lazard65}
M.~Lazard.
\newblock Groupes analytiques $p$-adiques.
\newblock {\em Publications Math\'ematiques de l'IH\'ES}, 26:5--219, 1965.

\bibitem{LMpowerful}
A.~Lubotzky and A.~Mann.
\newblock Powerful $p$-groups. i.
\newblock {\em Journal of Algebra}, 105:484--505, 1987.

\bibitem{LMpowerfulII}
A.~Lubotzky and A.~Mann.
\newblock Powerful $p$-groups. ii.
\newblock {\em Journal of Algebra}, 105:506--515, 1987.

\bibitem{ribeszalesskii13}
L.~Ribes and P.~Zalesskii.
\newblock {\em Profinite Groups}.
\newblock Ergebnisse der Mathematik und ihrer Grenzgebiete. 3. Folge / A Series
  of Modern Surveys in Mathematics. Springer Berlin Heidelberg, 2013.

\bibitem{schneider13}
P.~Schneider.
\newblock {\em $p$-Adic Lie Groups}.
\newblock Grundlehren der mathematischen Wissenschaften. Springer Berlin
  Heidelberg, 2013.

\bibitem{weibel94}
C.~A. Weibel.
\newblock {\em An {Introduction} to {Homological} {Algebra}}.
\newblock Cambridge {Studies} in {Advanced} {Mathematics} 38. Cambridge
  University Press, 1994.

\end{thebibliography}

\end{document}